\documentclass[a4paper,11pt]{amsart}%
\pdfoutput=1
\usepackage{hyperref}
\usepackage{amsfonts}
\usepackage{amsmath}
\usepackage{amssymb}
\usepackage{graphicx}
\usepackage{verbatim}%
\newtheorem{theorem}{Theorem}
\newtheorem{lemma}[theorem]{Lemma}
\newtheorem{proposition}[theorem]{Proposition}
\theoremstyle{definition}

\newtheorem{example}[theorem]{Example}

\newtheorem{remark}[theorem]{Remark}

\numberwithin{equation}{section}
\theoremstyle{plain}

\begin{document}
\title[Betti bounds for successive stellars]{Bounds for the Betti numbers of successive stellar subdivisions of a simplex}
\author{Janko B\"ohm}
\address{Janko B\"ohm\\
Department of Mathematics\\
University of Kaiserslautern\\
Erwin - Schr\"odinger - Str.\\
67663 Kaiserslautern\\
Germany}
\email{boehm@mathematik.uni-kl.de}
\author{Stavros Argyrios Papadakis}
\address{Centro de An\'{a}lise Matem\'{a}tica, Geometria e Sistemas Din\^{a}micos,
Departamento de Matem\'atica, Instituto Superior T\'ecnico, Universidade
T\'ecnica de Lisboa, Av. Rovisco Pais, 1049-001 Lisboa, Portugal}
\email{papadak@math.ist.utl.pt}
\thanks{S. P. was supported by the Portuguese Funda\c{c}\~ao para a Ci\^encia e a
Tecno\-lo\-gia through Grant SFRH/BPD/22846/2005 of POCI2010/FEDER and through
Project PTDC/MAT/099275/2008. He also benefited from a one month visit to
Technische Universit\"at Kaiserslautern financially supported by TU Kaiserslautern.}

\begin{abstract}
We give a bound for the Betti numbers of the Stanley-Reisner ring of a stellar
subdivision of a Gorenstein* simplicial complex by applying unprojection
theory. From this we derive a bound for the Betti numbers of iterated stellar
subdivisions of the boundary complex of a simplex. The bound depends only on
the number of subdivisions, and we construct examples which prove that it is sharp.

\end{abstract}
\maketitle

\section{Introduction}

Consider the class of simplicial complexes obtained from the boundary complex
of a simplex with $q+1$ vertices by any sequence of $c-1$ stellar
subdivisions. We give bounds for the (total) Betti numbers of the minimal
resolution of the associated Stanley-Reisner rings. The bounds depend only on
$c$ and not on $q$. Our main tool is the relation of stellar subdivisions of
Gorenstein* simplicial complexes with the Kustin-Miller complex construction
obtained in \cite{BP}, which gives an easy way to control the Betti numbers of
a stellar subdivision. By constructing a specific class of examples, we prove
that for fixed $c$ our bounds are attained for $q$ sufficiently large.

There are bounds in the literature for various classes of simplicial
complexes. If we only subdivide facets starting from a simplex the process
will yield a stacked polytope. In this case, there is an explicit formula for
the Betti numbers due to Terai and Hibi \cite{TH}. See also \cite{CK} for a
combinatorial proof, and \cite[Theorem 3.3]{HM}, \cite{BP} for the
construction of the resolutions. In \cite[Theorem 2.1, Proposition 3.4]{HM},
Herzog and Li Marzi consider bounds for a more general class than Gorenstein,
leading to a less sharp bound in our setting. Migliore and Nagel discuss in
\cite[Proposition 9.5]{MN} a bound for fixed $h$-vector. The bounds of
R\"{o}mer \cite{Roemer} apply for arbitrary ideals with a fixed number of
generators and linear resolution.

To state our results, for $c\geq1$ and $q\geq2$ denote by $\mathcal{D}_{q,c}$
the set of simplicial complexes $D$ on $q+c$ vertices which are obtained by
$c-1$ iterated stellar subdivisions of faces of positive dimension, starting
from the boundary complex of a $q$-simplex. If $k$ is any field, we denote by
$k[D]$ the Stanley-Reisner ring of $D$. Note, that $k[D]$ is the quotient of a
polynomial ring by a codimension $c$ Gorenstein ideal. Define inductively
$l_{c}=(l_{c,0},l_{c,1},\dots,l_{c,c})\in\mathbb{Z}^{c+1}$ by $l_{1}=\left(
1,1\right)  $ and
\[
l_{c}=2\left(  l_{c-1},0\right)  +2\left(  0,l_{c-1}\right)  -\left(
1,1,0,...,0\right)  -\left(  0,...,0,1,1\right)  \in\mathbb{Z}^{c+1}%
\]
for $c\geq2$. For example $l_{2}=2\left(  1,1,0\right)  +2\left(
0,1,1\right)  -\left(  1,1,0\right)  -\left(  0,1,1\right)  =(1,2,1)$,
$l_{3}=(1,5,5,1)$, and $l_{4}=(1,11,20,11,1)$. The main result of the paper is
the following theorem giving an upper bound for the Betti numbers of $k[D]$
for $D\in\mathcal{D}_{q,c}$. The bound follows immediately from the stronger
Theorem \ref{thmBettiBoundWithLink}, and that it is sharp from
Proposition~\ref{prop!ChampionsProposition}.

\begin{theorem}
\label{thm!bettiBoundWeakVersion}Suppose $c\geq1,q\geq2$ and $D\in
\mathcal{D}_{q,c}$. Then for the Betti numbers of $k[D]$ it holds that%
\begin{equation}
b_{i}(k[D])\leq l_{c,i} \label{equ!maintheorem}%
\end{equation}
for all $0\leq i\leq c$. Moreover, the bound is sharp in the following sense:
Given $c\geq1$, there exists $q\geq2$ and $F\in\mathcal{D}_{q,c}$ with
$b_{i}(k[F])=l_{c,i}$ for all $0\leq i\leq c$.
\end{theorem}

In Section \ref{sect!bound on the betti numbers} we focus on bounding the
Betti numbers of stellar subdivisions. The first result is Proposition
\ref{prop!bettiForStellarsVersion1}, which gives a bound for the Betti numbers
of the Stanley-Reisner ring of a stellar subdivision of a Gorenstein*
simplicial complex $D$ with respect to a face $\tau$ in terms of those of $D$
and of the link of $\tau$. The proof of this proposition uses Proposition
\ref{prop!stellar_for_generalised}, which is a generalization of \cite[Theorem
1.1]{BP}, and the Kustin-Miller complex construction (see \cite{KM} and
Section \ref{sec!kmcomplexconstrinstellar}). To prove Theorem
\ref{thm!bettiBoundWeakVersion} by induction on the codimension $c$, we have
to enlarge the class of complexes $\mathcal{D}_{q,c}$ by including also the
links of faces. We give a bound for their Betti numbers in Proposition
\ref{prop!bettiForStellarsVersion2}. According to the combinatorial Lemma
\ref{Lem_stellar_link} there are three types of links to consider.

Focussing on proving that the bound of Theorem \ref{thm!bettiBoundWeakVersion}
is sharp, we first analyze in Section \ref{sec!kmcomplexconstrinstellar} the
Kustin-Miller complex construction in the setting of stellar subdivisions. In
particular, we prove in Proposition
\ref{prop!about_minimality_with_multidegrees} a sufficient condition for the
minimality of the Kustin-Miller complex. In Section \ref{sect!champions} we
construct for any $c$ an element $F\in\mathcal{D}_{q,c}$ (for suitable $q$),
and using Proposition \ref{prop!about_minimality_with_multidegrees} we show
that the inequalities (\ref{equ!maintheorem}) are in fact equalities for $F$.
For an implementation of the construction see our package \textsc{BettiBounds}
\cite{BP4} for the computer algebra system \textsc{Macaulay2} \cite{GS}. Using
the minimality of the Kustin-Miller complex, we provide in the package a
function which efficiently produces the graded Betti numbers of the extremal
examples without the use of Gr\"{o}bner bases.

\section{Notation\label{sectnotation}}

For an ideal $I$ of a ring $R$ and $u\in R$ write $(I:u)=\{r\in R\bigm|ru\in
I\}$ for the ideal quotient. Denote by $\mathbb{N}$ the set of strictly
positive integer numbers. For $n\in\mathbb{N}$ we set $[n]=\{1,2,\dots,n\}$.
Assume $A\subset\mathbb{N}$ is a finite subset. We set $2^{A}$ to be the
simplex with vertex set $A$, by definition it is the set of all subsets of
$A$. A simplicial subcomplex $D\subset2^{A}$ is a subset with the property
that if $\tau\in D$ and $\sigma\subset\tau$ then $\sigma\in D$. The elements
of $D$ are also called faces of $D$, and the dimension of a face $\tau$ of $D$
is one less than the cardinality of $\tau$. We define the support of $D$ to
be
\[
\operatorname{supp}D=\{i\in A\bigm|\{i\}\in D\}\text{.}%
\]
We fix a field $k$. Denote by $R_{A}$ the polynomial ring $k[x_{a}\bigm|a\in
A]$ with the degrees of all variables $x_{a}$ equal to $1$. For a finitely
generated graded $R_{A}$-module $M$ we denote by $b_{i}(M)$ the $i$-th Betti
number of $M$, by definition $b_{i}(M)=\dim_{R_{A}/m}\operatorname{Tor}%
_{i}^{R_{A}}(R_{A}/m,M)$, where $m=(x_{a}\bigm|a\in A)$ is the maximal
homogeneous ideal of $R_{A}$. It is well-known that if we ignore shifts the
minimal graded free resolution of $M$ as $R_{A}$-module has the shape
\[
M\leftarrow R_{A}^{b_{0}(M)}\leftarrow R_{A}^{b_{1}(M)}\leftarrow R_{A}%
^{b_{2}(M)}\leftarrow\cdots
\]

For a simplicial subcomplex $D\subset2^{A}$ we define the Stanley-Reisner
ideal $I_{D,A}\subset R_{A}$ to be the ideal generated by the square free
monomials $x_{i_{1}}x_{i_{2}}\ldots x_{i_{p}}$ where $\{i_{1},i_{2}%
,\ldots,i_{p}\}$ is not a face of $D$. In particular, $I_{D,A}$ contains
linear terms if and only if $\operatorname{supp}D\not =A$. The Stanley-Reisner
ring $k[D,A]$ is defined by $k[D,A]=R_{A}/I_{D,A}$. Taking into account that
$\dim R_{A}=\#A$, we define the codimension of $k[D,A]$ by
$\operatorname{codim}k[D,A]=\#A-\dim k[D,A]$. For a nonempty face $\sigma$ of
$D$ we set $x_{\sigma}=\prod_{i\in\sigma}x_{i}\in k[D,A]$. We denote by
$b_{i}(k[D,A])$ the $i$-th Betti number of $k[D,A]$ considered as $R_{A}%
$-module. In the following, when the set $A$ is clear we will sometimes
simplify the notations $I_{D,A}$ to $I_{D}$ and $k[D,A]$ to $k[D]$. In some
situations, however, it will be convenient to consider Stanley-Reisner ideals
containing variables.

For a nonempty subset $A\subset\mathbb{N}$, we set $\partial A=2^{A}%
\setminus\{A\}\subset2^{A}$ to be the boundary complex of the simplex $2^{A}$.
For the Stanley-Reisner ring of $\partial A$ we have $k[\partial
A,A]=R_{A}/(\prod_{a\in A}x_{a})$.

Assume that, for $i=1,2$, $D_{i}\subset2^{A_{i}}$ is a subcomplex and the
subsets $A_{1},A_{2}$ of $\mathbb{N}$ are disjoint. By the join $D_{1}\ast
D_{2}$ of $D_{1}$ and $D_{2}$ we mean the subcomplex $D_{1}\ast D_{2}%
\subset2^{A_{1}\cup A_{2}}$ defined by
\[
D_{1}\ast D_{2}=\{\alpha_{1}\cup\alpha_{2}\bigm|\alpha_{1}\in D_{1},\alpha
_{2}\in D_{2}\}\text{.}%
\]
By \cite[p.~221, Exerc.~5.1.20]{BH} we have
\[
k[D_{1}\ast D_{2},A_{1}\cup A_{2}]=k[D_{1},A_{1}]\otimes_{k}k[D_{2}%
,A_{2}]\text{.}%
\]
As a consequence, using the well-known fact that the tensor product of the
minimal resolutions of two modules is a minimal resolution of the tensor
product of the modules we get that
\begin{equation}
b_{i}(k[D_{1}\ast D_{2}])=\sum_{t=0}^{i}b_{t}(k[D_{1}])b_{i-t}(k[D_{2}])
\label{eqn!betti_of_joint}%
\end{equation}
for all $i\geq0$.

If $\sigma$ is a face of $D\subset2^{A}$ define the link of $\sigma$ in $D$ to
be the subcomplex
\[
\operatorname*{lk}\nolimits_{D}\sigma=\{\alpha\in D\bigm|\alpha\cap
\sigma=\emptyset\text{ and }\alpha\cup\sigma\in D\}\subset2^{A\setminus\sigma
}\text{.}%
\]
It is clear that the Stanley-Reisner ideal of $\operatorname*{lk}%
\nolimits_{D}\sigma$ is equal to the intersection of the ideal $(I_{D,A}%
:x_{\sigma})$ with the subring $R_{A\setminus\sigma}$ of $R_{A}$. In other
words, it is the ideal of $R_{A\setminus\sigma}$ generated by the minimal
monomial generating set of $(I_{D,A}:x_{\sigma})$. Furthermore, define the
star of $\sigma$ in $D$ to be the subcomplex%
\[
\operatorname*{star}\nolimits_{D}\sigma=\{\alpha\in D\bigm|\alpha\cup\sigma\in
D\}\subset2^{A}\text{.}%
\]

If $\tau$ is a nonempty face of $D\subset2^{A}$ and $j\in\mathbb{N}\setminus
A$, we define the stellar subdivision $D_{\tau}$ with new vertex $j$ to be the
subcomplex
\[
D_{\tau}=\left(  D\setminus\operatorname*{star}\nolimits_{D}\tau\right)
\cup\left(  2^{\left\{  j\right\}  }\ast\operatorname*{lk}\nolimits_{D}%
\tau\ast\partial\tau\right)  \subset2^{A\cup\{j\}}\text{.}%
\]
Note that $D_{\tau}$ consists of the following faces:

\begin{enumerate}
\item All faces of $D$ which do not contain $\tau$.

\item For each face $\beta\in D$ with $\tau\subset\beta$ the faces
$(\beta\setminus\rho)\cup\{j\}$ for all nonempty subsets $\rho$ of $\tau$.
\end{enumerate}

\noindent It is easy to see that%
\begin{equation}
k[D_{\tau},2^{A\cup\{j\}}]=R_{A\cup\{j\}}/(I_{D,A},x_{\tau},x_{j}%
u\bigm|u\in(I_{D,A}:x_{\tau}))\text{.} \label{eqn!egns_of_stellar}%
\end{equation}

Following \cite[p.~67]{St}, we say that a subcomplex $D\subset2^{A}$ is
Gorenstein* over $k$ if $A=\operatorname{supp}D$, $k[D]$ is Gorenstein, and
for every $i\in A$ there exists $\sigma\in D$ with $\sigma\cup\{i\}$ not a
face of $D$. The last condition combinatorially means that $D$ is not a join
of the form $2^{\{i\}}\ast D_{1}$, and algebraically that $x_{i}$ divides at
least one element of the minimal monomial generating set of $I_{D,A}$. We say
that $D\subset2^{A}$ is generalized Gorenstein* over $k$ if $D\subset
2^{\operatorname{supp}D}$ is Gorenstein* over $k$. When there is no ambiguity
about the field $k$ we will just say Gorenstein* and generalized Gorenstein*.
It is well-known (cf.~ \cite[Section II.5]{St}) that if $D\subset2^{A}$ is
Gorenstein* and $\sigma\in D$ is a face then $\operatorname*{lk}%
\nolimits_{D}\sigma\subset2^{A\setminus\sigma}$ is also Gorenstein*. It
follows that if $D\subset2^{A}$ is generalized Gorenstein* and $\sigma\in D$
then $\operatorname*{lk}\nolimits_{D}\sigma\subset2^{A\setminus\sigma}$ is
also generalized Gorenstein*.

Recall also from \cite[Definition~1.2]{PR} that if $I=(f_{1},\dots
,f_{r})\subset R$ is a homogeneous codimension $1$ ideal of a graded
Gorenstein ring $R$ such that the quotient $R/I$ is Gorenstein, then there
exists $\psi\in\operatorname{Hom}_{R}(I,R)$ such that $\psi$ together with the
inclusion $I\hookrightarrow R$ generate $\operatorname{Hom}_{R}(I,R)$ as an
$R$-module. The Kustin--Miller unprojection ring of the pair $I\subset R$ is
defined as the quotient of $R[T]$ by the ideal generated by the elements
$Tf_{i}-\psi(f_{i})$, where $T$ is a new variable.

\section{Bounds for the Betti numbers of successive stellar
subdivisions\label{sect!bound on the betti numbers}}

The main result of this section is Theorem \ref{thmBettiBoundWithLink}, which
gives bounds for the Betti numbers of complexes in $\mathcal{D}_{q,c}$ and
links thereof.

In the following, let $D\subset2^{A}$ be a generalized Gorenstein* simplicial
complex, $\tau\in D$ a nonempty face and $D_{\tau}\subset2^{A\cup\{j\}}$ the
corresponding stellar subdivision with new vertex $j\in\mathbb{N}\setminus A$.
For simplicity set $\mathcal{R}=R_{A}[z]/(I_{D,A})=k[D,A][z]$, where $z$ is a
new variable.

In \cite{BP} we showed that a stellar subdivision of a face of a Gorenstein*
simplicial complex corresponds on the level of Stanley-Reisner rings to a
certain Kustin--Miller unprojection. In the following proposition we
generalize this statement for generalized Gorenstein* simplicial complexes.

\begin{proposition}
\label{prop!stellar_for_generalised}Assume that $\dim\tau\geq1$. Consider the
ideal $Q=(I_{D,A}:x_{\tau},$ $z)\subset R_{A}[z]$, and set
\[
M=\operatorname{Hom}_{\mathcal{R}}(Q/(I_{D,A}),\text{ }\mathcal{R})\text{.}%
\]
Then $M$ is generated, as $\mathcal{R}$-module, by the inclusion homomorphism
together with the map $\psi$ that sends $(I_{D,A}:x_{\tau})$ to $0$ and $z$ to
$x_{\tau}$. Denote by $S$ the Kustin--Miller unprojection ring of the pair
$Q/(I_{D,A})\subset\mathcal{R}$ associated to the map $\psi$. We have that $z$
is $S$-regular and $S/(z)\cong k[D_{\tau},A\cup\{j\}]$.
\end{proposition}

\begin{proof}
If $A=\operatorname{supp}D$ then the statement is \cite[Theorem 1.1(b)]{BP}.
Now assume that $\operatorname{supp}D$ is a proper subset of $A$. Consider
$P=\{x_{a}\bigm|a\in A\setminus\operatorname{supp}D\}\subset R_{A}$. We have
\[
I_{D,A}=(I_{D,\operatorname{supp}D})+(P),\quad Q=(I_{D,\operatorname{supp}%
D}:x_{\tau},z)+(P)
\]
and
\[
I_{D_{\tau},A\cup\{j\}}=(I_{D_{\tau},\operatorname{supp}D\cup\{j\}}%
)+(P)\text{.}%
\]
The arguments in the proof of \cite[Theorem 1.1]{BP} also prove that $M$ is
generated by the inclusion together with the map $\psi$ that sends
$(I_{D,A}:x_{\tau})$ to $0$ and $z$ to $x_{\tau}$. They also prove that $z$ is
$S$-regular and that $S/(z)\cong k[D_{\tau},A\cup\{j\}]$.
\end{proof}

We will now study the Betti numbers $b_{i}$ of $k[D_{\tau},A\cup\{j\}]$ as
$R_{A\cup\{j\}}$-module in terms of the Betti numbers of $k[D,A]$ as $R_{A}%
$-module and the Betti numbers of $k[\operatorname*{lk}\nolimits_{D}%
(\tau),A\setminus{\tau}]$ as $R_{A\setminus{\tau}}$-module.

\begin{proposition}
\label{prop!bettiForStellarsVersion1}Denote by $L=\operatorname*{lk}%
\nolimits_{D}(\tau)\subset2^{A\setminus\tau}$ the link of the face $\tau$ of
$D$. We then have
\[
b_{1}(k[D_{\tau}])\leq b_{1}(k[D])+b_{1}(k[L])+1
\]
and that, for $2\leq i\leq\operatorname{codim}k[D_{\tau}]-2$,
\[
b_{i}(k[D_{\tau}])\leq b_{i-1}(k[D])+b_{i}(k[D])+b_{i-1}(k[L])+b_{i}%
(k[L])\text{.}%
\]

\end{proposition}

\begin{proof}
If $\dim\tau=0$, say $\tau=\{i\}$, then
\[
I_{D_{\tau},A\cup\{j\}}=(G,x_{i})\text{,}%
\]
where $G$ is the finite set obtained by substituting $x_{j}$ for $x_{i}$ in
the minimal monomial generating set of $I_{D,A}$. Hence%
\[
b_{i}(k[D_{\tau}])=b_{i-1}(k[D])+b_{i}(k[D])
\]
for all $i$.

Now assume that $\dim\tau\geq1$. Using the notations of
Proposition~\ref{prop!stellar_for_generalised}, we have that $S$ is the
Kustin--Miller unprojection of the pair $Q/(I_{D,A})\subset\mathcal{R}$ and
that $b_{i}(k[D_{\tau}])=b_{i}[S]$ for all $i$.

We denote by $C_{U}$ the graded free resolution of $S$ obtained by the
Kustin--Miller complex construction,
cf.~Section~\ref{sec!kmcomplexconstrinstellar} and \cite[Section~2]{BP3}, with
initial data the minimal graded free resolutions of $\mathcal{R}%
=R_{A}[z]/(I_{D,A})$ and $R_{A}[z]/Q$ over $R_{A}[z]$. Since $C_{U}$ is a
graded free resolution of $S$ we have $b_{i}[S]\leq b_{i}(C_{U})$ for all $i$,
where $b_{i}(C_{U})$ denotes the rank of the finitely generated free
$R_{A}[z]$-module $(C_{U})_{i}$. The variable $z$ does not appear in the
minimal monomial generating set of $I_{D,A}$, as a consequence $b_{i}%
(\mathcal{R})=b_{i}(k[D])$ for all $i$. Since $Q=(I_{D,A}:x_{\tau},z)$ and the
variable $z$ does not appear in the minimal generating set of $(I_{D,A}%
:x_{\tau})$ we have for all $i$
\begin{align}
\phantom{===}b_{i}(R_{A}[z]/Q)  &  =b_{i-1}(R_{A}/(I_{D,A}:x_{\tau}%
))+b_{i}(R_{A}/(I_{D,A}:x_{\tau}))\label{eqn!yutdg}\\
&  =b_{i-1}(k[L])+b_{i}(k[L]).\nonumber
\end{align}
Moreover, by the Kustin--Miller complex construction (see \cite[Section~2]%
{BP3}) we have
\[
b_{1}(C_{U})\leq b_{1}(k[D])+b_{1}(R_{A}[z]/Q)=b_{1}(k[D])+b_{1}(k[L])+1
\]
and, for $2\leq i\leq\operatorname{codim}k[D_{\tau}]-2$, that
\[
b_{i}(C_{U})\leq b_{i-1}(k[D])+b_{i}(R_{A}[z]/Q)+b_{i}(k[D]).
\]
Hence
\[
b_{i}(k[D_{\tau}])=b_{i}[S]\leq b_{i}(C_{U})\leq b_{i-1}(k[D])+b_{i}%
(R_{A}[z]/Q)+b_{i}(k[D])
\]
which combined with Equality(\ref{eqn!yutdg}) finishes the proof.
\end{proof}

\begin{remark}
It may be interesting to investigate, perhaps with the use of Hochster's
formula or a generalization of the Kustin-Miller complex technique, whether
the inequalities of Proposition \ref{prop!bettiForStellarsVersion1} hold in a
more general setting than Gorenstein*.
\end{remark}

For the proof of Proposition \ref{prop!bettiForStellarsVersion2} we will need
the following combinatorial lemma which relates a link of a stellar
subdivision with links of the original simplicial complex. The straightforward
but lengthy proof will be given in
Subsection~\ref{subs!proofofCombinatorialLemma}.

\begin{lemma}
\label{Lem_stellar_link}If $\sigma$ is a nonempty face of $D_{\tau}$ the
following hold:

\begin{enumerate}
\item (Case I) Assume $j\notin\sigma$ and $\tau\cup\sigma\in D$. Then
$\tau\setminus\sigma$ is a nonempty face of $\operatorname*{lk}\nolimits_{D}%
\sigma$ and
\[
\operatorname*{lk}\nolimits_{D_{\tau}}{\sigma}=(\operatorname*{lk}%
\nolimits_{D}\sigma)_{\tau\setminus\sigma}%
\]
that is, $\operatorname*{lk}\nolimits_{D_{\tau}}{\sigma}$ is the stellar
subdivision of $\operatorname*{lk}\nolimits_{D}\sigma$ with respect to
$\tau\setminus\sigma$.

\item (Case II) Assume that $j\notin\sigma$ and $\tau\cup\sigma\notin D$. Then
$\operatorname*{lk}\nolimits_{D_{\tau}}{\sigma}$ is equal to
$\operatorname*{lk}\nolimits_{D}\sigma$ considered as a subcomplex of
$2^{(A\cup\{j\})\setminus\sigma}$.

\item (Case III) Assume $j\in\sigma$. Then $\tau\cup\sigma\setminus\{j\}$ is a
face of $D$, $\tau\setminus\sigma$ is nonempty and
\[
\operatorname*{lk}\nolimits_{D_{\tau}}{\sigma}=\operatorname*{lk}%
\nolimits_{D}(\tau\cup\sigma\setminus\{j\})\ast\partial(\tau\setminus\sigma)
\]
that is, $\operatorname*{lk}\nolimits_{D_{\tau}}{\sigma}$ is equal to the join
of $\operatorname*{lk}\nolimits_{D}(\tau\cup\sigma\setminus\{j\})$ with
$\partial(\tau\setminus\sigma)$.
\end{enumerate}
\end{lemma}

\begin{remark}
Case II corresponds to faces ${\sigma}$ of $D\setminus\operatorname*{star}%
\nolimits_{D}\tau$, while Cases I and III to faces of $2^{\left\{  j\right\}
}\ast\operatorname*{lk}\nolimits_{D}\tau\ast\partial\tau$.

\end{remark}

The next proposition gives bounds on the Betti numbers of links of a stellar
subdivision in terms of links of the original complex.

\begin{proposition}
\label{prop!bettiForStellarsVersion2} Let $\sigma$ be a face of $D_{\tau}$,
and set $L=\operatorname*{lk}\nolimits_{D_{\tau}}\sigma\subset2^{(A\cup
\{j\})\setminus\sigma}$.

\begin{enumerate}
\item (Case I) If $j\not \in \sigma$ and $\tau\cup\sigma$ is a face of $D$
then we have that
\[
b_{1}(k[L])\leq b_{1}(k[L_{1}])+b_{1}(k[L_{2}])+1
\]
and that for $2\leq i\leq\operatorname{codim}k[L]-2$%
\[
b_{i}(k[L])\leq b_{i-1}(k[L_{1}])+b_{i}(k[L_{1}])+b_{i-1}(k[L_{2}%
])+b_{i}(k[L_{2}])\text{,}%
\]
where $L_{1}=\operatorname*{lk}\nolimits_{D}\sigma\subset2^{A\setminus\sigma}$
and $L_{2}=\operatorname*{lk}\nolimits_{D}(\tau\cup\sigma)\subset
2^{A\setminus(\tau\cup\sigma)}$.

\item (Case II) If $j\not \in \sigma$ and $\tau\cup\sigma$ is not a face of
$D$ then we have that for all $i$%
\[
b_{i}(k[L])=b_{i-1}(k[L_{1}])+b_{i}(k[L_{1}])\text{.}%
\]

\item (Case III) Assume $j\in\sigma$. Then $\tau\cup\sigma\setminus\{j\}$ is a
face of $D$ and we have that for all $i$%
\[
b_{i}(k[L])=b_{i-1}(k[L_{3}])+b_{i}(k[L_{3}])\text{,}%
\]
where $L_{3}=\operatorname*{lk}\nolimits_{D}(\tau\cup\sigma\setminus
\{j\})\subset2^{A\setminus(\tau\cup\sigma)}$.
\end{enumerate}
\end{proposition}

\begin{proof}
Assume first we are in Case I, that is $j\not \in \sigma$ and $\tau\cup\sigma$
is a face of $D$. By part (1) of Lemma~\ref{Lem_stellar_link} we have
$L=(L_{1})_{\tau\setminus\sigma}$. Furthermore, a straightforward calculation
shows that $\operatorname*{lk}\nolimits_{D}(\tau\cup\sigma)=\operatorname*{lk}%
\nolimits_{L_{1}}(\tau\setminus\sigma)$. 
The result follows from Proposition~\ref{prop!bettiForStellarsVersion1}
applied to the stellar subdivision of the face $\tau\setminus\sigma$ of
$L_{1}$.

Assume now we are in Case II, that is $j\not \in \sigma$ and $\tau\cup\sigma$
is not a face of $D$. By part (2) of Lemma~\ref{Lem_stellar_link} we have
\[
I_{L,(A\cup\{j\})\setminus\sigma}=(I_{L_{1},A\setminus\sigma})+(x_{j})\subset
R_{(A\cup\{j\})\setminus\sigma}%
\]
and the result is clear. In Case III, that is, $j\in\sigma$, we have by part
(3) of Lemma~\ref{Lem_stellar_link} that $L=L_{3}\ast\partial(\tau
\setminus\sigma)$. Since $k[\partial(\tau\setminus\sigma)]$ is the quotient of
a polynomial ring by a single equation, hence has nonzero Betti numbers only
$b_{0}=b_{1}=1$, the result follows by Equation~(\ref{eqn!betti_of_joint}).
\end{proof}

For $c\geq1$ and $q\geq2$ recall that we defined $\mathcal{D}_{q,c}$ as the
set of simplicial subcomplexes $D\subset2^{[q+c]}$ such that there exists a
sequence of simplicial complexes
\[
D_{1},D_{2},\dots,D_{c-1},D_{c}=D
\]
with the property that $D_{1}=\partial([q+1])\subset2^{[q+1]}$ is the boundary
complex of the simplex on $q+1$ vertices, and, for $0\leq i\leq c-1$,
$D_{i+1}\subset2^{[q+i+1]}$ is obtained from $D_{i}\subset2^{[q+i]}$ by a
stellar subdivision of a face of $D_{i}$ of dimension at least $1$ with new
vertex $q+i+1$. It is clear that $\operatorname{supp}D_{i}=[q+i]$ and
$\operatorname{codim}k[D_{i}]=i$ for all $i$.

Assume $D\in\mathcal{D}_{q,c}$ and consider the Stanley--Reisner ring
$k[D]=R_{[q+c]}/I_{D}$. By \cite[Corollary~5.6.5]{BH} $D$ is Gorenstein*. As a
consequence, since $\operatorname{codim}k[D]=c$ the only nonzero Betti numbers
$b_{i}$ of $k[D]$ are $1=b_{0},b_{1},\dots,b_{c-1},b_{c}=1$ and $b_{i}%
=b_{c-i}$ for all $i$.

To prove Theorem \ref{thm!bettiBoundWeakVersion} we need to enlarge the class
of ideals we consider by including the ideals of links. For $q\geq2$ and
$c\geq1$ we define
\begin{align*}
\mathcal{I}_{q,c}  &  =\{I_{D}\ \bigm|\ D\in\mathcal{D}_{q,c}\}\;\;\cup\\
&  \{\left(  I_{D}:x_{\sigma}\right)  \subset k[x_{1},\dots,x_{q+c}%
]\bigm|D\in\mathcal{D}_{q,c},\;\sigma\in D\text{ a nonempty face}\}.
\end{align*}
The following theorem is the key technical result.

\begin{theorem}
\label{thmBettiBoundWithLink}Suppose $c\geq1$, $q \geq2$, $R=k[x_{1},
\dots,x_{q+c}]$ and $I\in\mathcal{I}_{q,c}$. Then for the the Betti numbers
$b_{i}\left(  R/I\right)  $ of the minimal resolution of $R/I$ as $R$-module
it holds that%
\[
b_{i}\left(  R/I\right)  \leq l_{c,i}%
\]
for all $0 \leq i \leq c$.
\end{theorem}

\begin{proof}
First note that from the definition of the bounding sequence $l_{c}$ it is
clear that $l_{c,i}=l_{c,c-i}$ for all $0\leq i\leq c$; that $l_{c,0}%
=l_{c,c}=1$ for all $c\geq1$; that $l_{c+1,1}=2l_{c,1}+1$ for all $c\geq2$;
and that $l_{c+1,i}=2l_{c,i-1}+2l_{c,i}$ for $2\leq i\leq(c+1)-2$.

Assume the claim is not true. Then there exist $c\geq1$, $q\geq2$ and an ideal
$I\in\mathcal{I}_{q,c}$ with $I\notin\mathcal{W}_{c}$, where by definition
\[
\mathcal{W}_{c}=\{I\in\mathcal{I}_{p,c}\bigm|p\geq2\;\text{ and }%
\;b_{i}\left(  R/I\right)  \leq l_{c,i}\;\text{ for all }i\}.
\]
We fix such an ideal $I$ with $c$ the least possible, and we will get a
contradiction. Since Gorenstein codimension $1$ or $2$ implies complete
intersection, we necessarily have $c\geq3$.

The first case is that $I=I_{D_{1}}$ for some $D_{1}\in\mathcal{D}_{q,c}$, so
there exists $D\in\mathcal{D}_{q,c-1}$ and a face $\tau$ of $D$ of dimension
at least $1$ such that $D_{1}=D_{\tau}$. Since $c$ has been chosen to be the
smallest possible, we have that $I_{D}\in\mathcal{W}_{c-1}$ and $(I_{D}%
:x_{\tau})\in\mathcal{W}_{c-1}$. Using the properties of $l_{c}$ mentioned
above, it follows by Proposition~\ref{prop!bettiForStellarsVersion1} that
$I\in\mathcal{W}_{c}$, which is a contradiction.

Assume now that $I=(I_{D_{1}}:x_{\sigma})$ for some $D_{1}\in\mathcal{D}%
_{q,c}$ and face $\sigma$ of $D_{1}$. Write $D_{1}=D_{\tau}$, for some
$D\in\mathcal{D}_{q,c-1}$ and face $\tau$ of $D$ of dimension at least $1$.
The new vertex $j$ of $D_{\tau}$ is $q+c$. In the remaining of the proof we
will use the simplicial complexes $L$ and $L_{i}$, with $1\leq i\leq3$,
defined in Proposition~\ref{prop!bettiForStellarsVersion2}. We have three
cases. For all of them we will show that $I\in\mathcal{W}_{c}$, which is a contradiction.

Assume we are in Case I, that is $j\not \in \sigma$ and $\tau\cup\sigma$ is a
face of $D$. Since by the minimality of $c$ we have that both ideals $I_{L}$
and $I_{L_{1}}$ are in $\mathcal{W}_{c-1}$, it follows by Case I of
Proposition~\ref{prop!bettiForStellarsVersion2} that $I\in\mathcal{W}_{c}$.
Assume now we are in Case II, that is $j\not \in \sigma$ and $\tau\cup\sigma$
is not a face of $D$. Again by the minimality of $c$ we have $I_{L_{1}}%
\in\mathcal{W}_{c-1}$, so using Case II of
Proposition~\ref{prop!bettiForStellarsVersion2} it follows that $I\in
\mathcal{W}_{c}$. Finally, assume we are in Case III, that is $j\in\sigma$. By
the minimality of $c$ we have $I_{L_{3}}\in\mathcal{W}_{c-1}$, so using Case
III of Proposition~\ref{prop!bettiForStellarsVersion2} it follows that
$I\in\mathcal{W}_{c}$. This finishes the proof.
\end{proof}

\begin{remark}
Combining Proposition~\ref{prop!bettiForStellarsVersion2} with
Theorem~\ref{thmBettiBoundWithLink} it is not hard to show that for fixed
$q\geq2$, there exists $c_{0}\geq1$ such that $b_{i}(k[D])<l_{c,i}$ for all
$c\geq c_{0},D\in\mathcal{D}_{q,c}$ and $1\leq i\leq c-1$. So if we fix $q$
for $c$ sufficiently large the Betti bound in
Theorem~\ref{thm!bettiBoundWeakVersion} is not sharp. We leave the details to
the interested reader.
\end{remark}

\subsection{Proof of Lemma~\ref{Lem_stellar_link}}

\label{subs!proofofCombinatorialLemma}

We will repeatedly use in the following two observations: If $\alpha\in D$, we
have $\alpha\in D_{\tau}$ if and only if $\tau$ is not a subset of $\alpha$.
Moreover, if $\beta\in D_{\tau}$ and $j\notin\beta$ then $\beta\in D$. We will
also use the following notation. For $\beta\in D$ with $\tau\subset\beta$ and
nonempty $\rho\subset\tau$ we set
\[
\operatorname{transf}(\beta,\rho)=(\beta\setminus\rho)\cup\{j\}\in D_{\tau}.
\]
Using this notation, $D_{\tau}$ is the disjoint union of the set consisting of
the faces $\alpha\in D$ which do not contain $\tau$ with the set
\[
\{\operatorname{transf}(\beta,\rho)\bigm|\beta\in D\text{ with }\tau
\subset\beta,\;\emptyset\not =\rho\subset\tau\}.
\]
Note, that $\tau\setminus\sigma$ is nonempty since $\sigma\in D_{\tau}$
implies that $\tau$ is not a subset of $\sigma$. Recall that%
\[
\operatorname{lk}_{D}{\sigma=}\{\alpha\in D\bigm|\alpha\cap\sigma
=\emptyset\text{ and }\alpha\cup\sigma\in D\}\text{.}%
\]
\smallskip

\emph{CASE I:} Assume $j\notin\sigma$ and $\sigma\cup\tau\in D$. Since
$\sigma\cap(\tau\setminus\sigma)=\emptyset$ and $\sigma\cup\tau\in D$ we have
that indeed $\tau\setminus\sigma$ is a face of $\operatorname{lk}_{D}\sigma$.
We prove that $\operatorname{lk}_{D_{\tau}}{\sigma}=(\operatorname{lk}%
_{D}\sigma)_{\tau\setminus\sigma}$.

Given $\alpha\in\operatorname{lk}_{D_{\tau}}{\sigma}$, we show that $\alpha
\in(\operatorname{lk}_{D}\sigma)_{\tau\setminus\sigma}$. There are two subcases:

\emph{Subcase 1.1:} Assume $j\notin\alpha$. Then $\alpha\cup\sigma\in D_{\tau
}$ and $j\notin\alpha\cup\sigma$ implies $\alpha\cup\sigma\in D$, hence by
$\alpha\cap\sigma=\emptyset$ we have $\alpha\in\operatorname{lk}_{D}\sigma$.
If $\tau\setminus\sigma$ is a subset of $\alpha$ we get $\tau\subset\alpha
\cup\sigma$, which contradicts $\alpha\cup\sigma\in D_{\tau}$. So
$\tau\setminus\sigma$ is not a subset of $\alpha$, which implies that
$\alpha\in(\operatorname{lk}_{D}\sigma)_{\tau\setminus\sigma}$.

\emph{Subcase 1.2:} Assume $j\in\alpha$. Since $\alpha\cup\sigma\in D_{\tau}$
there exist $\beta\in D$ with $\tau\subset\beta$ and nonempty $\rho\subset
\tau$ such that
\begin{equation}
\alpha\cup\sigma=\operatorname{transf}(\beta,\rho)=(\beta\setminus\rho
)\cup\{j\}. \label{eqn!temp9654sw}%
\end{equation}
As a consequence, using $\alpha\cap\sigma=\emptyset$, we get $\alpha
=(\beta\setminus(\rho\cup\sigma))\cup\{j\}$. Since $j\notin\sigma$
Equation~(\ref{eqn!temp9654sw}) also implies $\sigma\subset\beta$. Set
$\beta^{\prime}=\beta\setminus\sigma\in D$. It is enough to show that
$\rho,\beta^{\prime}\in\operatorname{lk}_{D}\sigma$, $\tau\setminus
\sigma\subset\beta^{\prime}$, $\emptyset\neq\rho\subset\tau\setminus\sigma$,
and
\begin{equation}
\alpha=(\beta^{\prime}\setminus\rho)\cup\{j\}=\operatorname{transf}%
(\beta^{\prime},\rho). \label{equ alpha transf}%
\end{equation}
By Equation~(\ref{eqn!temp9654sw}), we have $\rho\cap\sigma=\emptyset$. By
definition, $\beta^{\prime}\cap\sigma=\emptyset$. Moreover, $\rho\cup
\sigma\subset\tau\cup\sigma\in D$, and $\beta^{\prime}\cup\sigma=\beta\in D$.
Since $\tau\subset\beta$ we have $\tau\setminus\sigma\subset\beta^{\prime}$.
By $\rho\subset\tau$ and $\rho\cap\sigma=\emptyset$ it follows that
$\rho\subset\tau\setminus\sigma$. Finally, Equation~(\ref{equ alpha transf})
follows from Equation~(\ref{eqn!temp9654sw}) using $\alpha\cap\sigma
=\emptyset$ and $j\notin\sigma$.

Conversely, assume $\alpha\in(\operatorname{lk}_{D}\sigma)_{\tau
\setminus\sigma}$, that is, $\alpha$ is in the stellar of the link. We will
prove that $\alpha\in\operatorname{lk}_{D_{\tau}}\sigma$. We have two subcases:

\emph{Subcase 2.1:} Assume $j\notin\alpha$. Then $\alpha\in\operatorname{lk}%
_{D}\sigma$. We have that $\tau\setminus\sigma$ is not a subset of $\alpha$
(since $\tau\setminus\sigma$ a subset of $\alpha$ implies $\alpha$ not in
$(\operatorname{lk}_{D}\sigma)_{\tau\setminus\sigma}$, a contradiction), as a
consequence $\tau$ is not a subset of $\alpha\cup\sigma$. Hence $\alpha
\cup\sigma\in D_{\tau}$ which implies that $\alpha\in\operatorname{lk}%
_{D_{\tau}}\sigma$.

\emph{Subcase 2.2:} Assume $j\in\alpha$. Then there exist $\beta
\in\operatorname{lk}_{D}\sigma$ with $\tau\setminus\sigma\subset\beta$ and
nonempty $\rho\subset\tau\setminus\sigma$ with $\alpha=\operatorname{transf}%
(\beta,\rho)=(\beta\setminus\rho)\cup\{j\}$. To finish the proof of the
subcase we will show that $\beta^{\prime}=\beta\cup\sigma$ is a face of $D$
containing $\tau$ and $\alpha\cup\sigma=\operatorname{transf}(\beta^{\prime
},\rho)$. Indeed, $\beta\in\operatorname{lk}_{D}\sigma$ implies $\beta
^{\prime}\in D$, and $\tau\setminus\sigma\subset\beta$ implies $\tau
\subset\beta^{\prime}$. Moreover, since $\rho\cap\sigma=\emptyset$ we have
\[
\operatorname{transf}(\beta^{\prime},\rho)=(\beta^{\prime}\setminus{\rho}%
)\cup\{j\}=\sigma\cup((\beta\setminus{\rho})\cup\{j\})=\alpha\cup\sigma,
\]
which finishes the proof of CASE I. \smallskip

\emph{CASE II:} Assume $j\notin\sigma$ and $\tau\cup\sigma\notin D$. We prove
that $\operatorname{lk}_{D_{\tau}}{\sigma}=\operatorname{lk}_{D}\sigma$.

Given $\alpha\in\operatorname{lk}_{D_{\tau}}{\sigma}$, we show that $\alpha
\in\operatorname{lk}_{D}\sigma$. There are two subcases (in fact, we will show
the second cannot happen):

\emph{Subcase 1.1:} Assume $j\notin\alpha$. This implies $j\notin(\alpha
\cup\sigma)$ hence $\alpha\cup\sigma\in D$. Therefore $\alpha\in
\operatorname{lk}_{D}\sigma$.

\emph{Subcase 1.2:} Assume $j\in\alpha$. Then there exist a face $\beta$ of
$D$ with $\tau\subset\beta$ and nonempty $\rho\subset\tau$ such that
\[
\alpha\cup\sigma=\operatorname{transf}(\beta,\rho)=(\beta\setminus\rho
)\cup\{j\}.
\]
Hence, $(\alpha\cup\sigma)\backslash\{j\}=\beta\setminus\rho$, which implies
\[
\tau\cup\sigma\subset\tau\cup(\alpha\cup\sigma)\backslash\{j\}\subset\tau
\cup\beta=\beta\in D\text{.}%
\]
From this it follows that $\tau\cup\sigma\in D$, contradicting the assumption
$\tau\cup\sigma\notin D$. So $j\in\alpha$ is impossible.

Conversely, assume $\alpha\in\operatorname{lk}_{D}\sigma$. To show $\alpha
\in\operatorname{lk}_{D_{\tau}}{\sigma}$ it is enough to prove $\alpha
\cup\sigma\in D_{\tau}$, which follows from $\tau\not \subset \alpha\cup
\sigma$. So assume $\tau\subset\alpha\cup\sigma$, then $\tau\setminus
\sigma\subset\alpha\in\operatorname{lk}_{D}\sigma$ so $(\tau\setminus
\sigma)\cup\sigma\in D$, hence $\tau\cup\sigma\in D$, contradicting the
assumption $\tau\cup\sigma\notin D$. So $\tau$ is not a subset of $\alpha
\cup\sigma$. This finishes the proof of CASE II. \smallskip

\emph{CASE III:} We assume $j\in\sigma$. We first show that $\tau\cup
\sigma\setminus\{j\}$ is a face of $D$. Indeed, $\sigma\in D_{\tau}$ and
$j\in\sigma$ imply that there exist a face $\beta_{1}$ of $D$ with
$\tau\subset\beta_{1}$ and nonempty $\rho_{1}\subset\tau$ such that
\[
\sigma=\operatorname{transf}(\beta_{1},\rho_{1})=(\beta_{1}\setminus\rho
_{1})\cup\{j\}.
\]
As a consequence $\sigma\setminus\{j\}\subset\beta_{1}$ which together with
$\tau\subset\beta_{1}$ implies that $\tau\cup\sigma\setminus\{j\}\subset
\beta_{1}$, hence $\tau\cup\sigma\setminus\{j\}$ is a face of $D$. We will
show that
\[
\operatorname{lk}_{D_{\tau}}{\sigma}=\operatorname{lk}_{D}(\tau\cup
\sigma\setminus\{j\})\ast\partial(\tau\setminus\sigma).
\]

Assume $\alpha\in\operatorname{lk}_{D_{\tau}}{\sigma}$. Then $\alpha\cap
\sigma=\emptyset$, hence $j\notin\alpha$. Since $j\in\alpha\cup\sigma$ and
$\alpha\cup\sigma\in D_{\tau}$ there exists $\beta\in D$ with $\tau
\subset\beta$ and nonempty $\rho\subset\tau$ such that
\[
\alpha\cup\sigma=\operatorname{transf}(\beta,\rho)=(\beta\setminus\rho
)\cup\{j\},
\]
so in particular $\sigma\setminus\{j\}\subset\beta$ and $(\alpha\cup
\sigma)\cap\rho=\emptyset$.

Set $\alpha_{1}=\alpha\cap(\tau\setminus\sigma)$ and $\alpha_{2}%
=\alpha\setminus\alpha_{1}$, hence $\alpha_{2}\cap(\tau\setminus
\sigma)=\emptyset$. Since $\alpha$ is the (disjoint) union of $\alpha_{1}$ and
$\alpha_{2}$ we need to show that $\alpha_{1}\in\partial(\tau\setminus\sigma)$
and $\alpha_{2}\in\operatorname{lk}_{D}(\tau\cup\sigma\setminus\{j\})$. If
$\alpha_{1}=\tau\setminus\sigma$ we would have $(\tau\setminus\sigma
)\subset\alpha$, hence $\tau\subset(\alpha\cup\sigma)$ which contradicts that
$(\alpha\cup\sigma)\cap\rho=\emptyset$. Hence $\alpha_{1}\in\partial
(\tau\setminus\sigma)$.

Since $\alpha\cap\sigma=\emptyset$ we get $\alpha_{2}\cap\sigma=\emptyset$,
which together with $\alpha_{2}\cap(\tau\setminus\sigma)=\emptyset$ implies
that $\alpha_{2}\cap(\tau\cup\sigma\setminus\{j\})=\emptyset$. We will show
$\alpha_{2}\cup(\tau\cup\sigma\setminus\{j\})\in D$. Since $\alpha\subset
\beta\cup\{j\}$ and $j\notin\alpha$, we have $\alpha\subset\beta$, hence
$\alpha_{2}\subset\beta$. By the definition of $\beta$ we have $\tau
\subset\beta$ and as we showed above $\sigma\setminus\{j\}\subset\beta$. As a
consequence $(\alpha_{2}\cup\tau\cup\sigma)\setminus\{j\}\subset\beta$, hence
$(\alpha_{2}\cup\tau\cup\sigma)\setminus\{j\}\in D$. This finishes the proof
of $\alpha\in\operatorname{lk}_{D}(\tau\cup\sigma\setminus\{j\})\ast
\partial(\tau\setminus\sigma)$.\medskip

For the converse, assume $\alpha_{1}\in\partial(\tau\setminus\sigma)$ and
$\alpha_{2}\in\operatorname{lk}_{D}(\tau\cup\sigma\setminus\{j\})$. We will
show that $\alpha_{1}\cup\alpha_{2}\in\operatorname{lk}_{D_{\tau}}\sigma$. We
have that $(\alpha_{2}\cup\tau\cup\sigma)\setminus\{j\}\in D$, that
$(\alpha_{2}\cap(\tau\cup\sigma)\setminus\{j\})=\emptyset$ (in particular
$\alpha_{2}\cap\tau=\emptyset$ and $\alpha_{2}\cap\sigma=\emptyset$ since
$j\notin\alpha_{2}$), that $\alpha_{1}\cap\sigma=\emptyset$ and that
$\alpha_{1}$ is a proper subset of $\tau\setminus\sigma$. Hence there exists
$\gamma\in(\tau\setminus\sigma)\setminus\alpha_{1}=\tau\setminus(\alpha
_{1}\cup\sigma)$. Taking into account that $\alpha_{2}\cap\tau=\emptyset$ it
follows that $\gamma\in\tau\setminus(\alpha_{1}\cup\alpha_{2}\cup\sigma)$.

Since $\alpha_{2}\cap\tau=\emptyset$ and $\alpha_{1}\subset\tau$ we have
$\alpha_{1}\cap\alpha_{2}=\emptyset$. We will now show that $\alpha_{1}%
\cup\alpha_{2}\in\operatorname{lk}_{D_{\tau}}\sigma$. First as we observed
above both $\alpha_{1}$ and $\alpha_{2}$ have empty intersection with $\sigma
$. So it is enough to show that $(\alpha_{1}\cup\alpha_{2}\cup\sigma)\in
D_{\tau}$. Set $\beta_{2}=(\alpha_{2}\cup\tau\cup\sigma)\setminus\{j\}$,
which, as observed above, is in $D$. Since $\gamma\in\tau$, it follows that
$\operatorname{transf}(\beta_{2},\{\gamma\})\in D_{\tau}$. Since, as observed
above, $\gamma\in\tau\setminus(\alpha_{1}\cup\alpha_{2}\cup\sigma)$ we have
\[
(\alpha_{1}\cup\alpha_{2}\cup\sigma)\subset(\alpha_{2}\cup\tau\cup
\sigma)\setminus\{\gamma\}=\operatorname{transf}(\beta_{2},\{\gamma\}),
\]
hence $(\alpha_{1}\cup\alpha_{2}\cup\sigma)\in D_{\tau}$. This finishes the
proof of CASE III, and hence the proof of Lemma~\ref{Lem_stellar_link}.

\section{The structure of the Kustin--Miller complex in the stellar
subdivision case\label{sec!kmcomplexconstrinstellar}}

Kustin and Miller introduced in \cite{KM} the Kustin--Miller complex
construction which produces a projective resolution of the Kustin--Miller
unprojection ring in terms of projective resolutions of the initial data. In
Proposition \ref{prop!about_minimality_with_multidegrees} we prove a criterion
for the minimality of the resolution, which will be used in
Section~\ref{sect!champions}. For that, we analyze the additional structure of
the construction in the case of stellar subdivisions.

We will use the graded version of the Kustin--Miller complex construction as
described in \cite[Section~2]{BP3}. Note, that there is an implementation of
the construction available for the computer algebra system \textsc{Macaulay2},
see [loc. cit.].

In this section $D\subset2^{A}$ will be a generalized Gorenstein* simplicial
complex, $\tau\in D$ a face of positive dimension and $D_{\tau}\subset
2^{A\cup\{j\}}$ the corresponding stellar subdivision with new vertex
$j\in\mathbb{N}\setminus A$.

Let $R=R_{A}[z]$ with the following grading: $\deg x_{a}=1$ for $a\in A$ and
$\deg z=\dim\tau$. Write $I\subset R$ for the ideal generated by $I_{D,A}$ and
set $J=(I_{D,A}:x_{\tau},z)\subset R$. Denote by%
\[%
\begin{tabular}
[c]{ll}%
$C_{J}:$ & $R/J\leftarrow A_{0}\overset{a_{1}}{\leftarrow}A_{1}\overset{a_{2}%
}{\leftarrow}\dots\overset{a_{g-1}}{\leftarrow}A_{g-1}\overset{a_{g}%
}{\leftarrow}A_{g}\leftarrow0$\\
$C_{I}:$ & $R/I\leftarrow B_{0}\overset{b_{1}}{\leftarrow}B_{1}\overset{b_{2}%
}{\leftarrow}\dots\overset{b_{g-1}}{\leftarrow}B_{g-1}\leftarrow0$%
\end{tabular}
\]
the minimal graded free resolutions of $R/J$ and $R/I$ respectively.

By Proposition~\ref{prop!stellar_for_generalised} $\operatorname{Hom}%
_{R/I}(J/I,R/I)$ is generated as an $R/I$-module by the inclusion homomorphism
together with the map $\psi$ that sends $(I_{D,A}:x_{\tau})$ to $0$ and $z$ to
$x_{\tau}$. By the Kustin-Miller complex construction we obtain the
unprojection ideal $U\subset R[T]$ of the pair $J/I\subset R/I$ defined by
$\psi$ with new variable $T$, and a, in general non-minimal, graded free
resolution $C_{U}$ of $R[T]/U$ as $R[T]$-module. For more details see [loc. cit.].

Clearly, the $k$-algebra $S$ defined in
Proposition~\ref{prop!stellar_for_generalised} is isomorphic to $R[T]/U$,
since it is obtained from $R[T]/U$ by substituting $T$ with $x_{j}$. By the
same proposition $z$ is $R[T]/U$-regular and $(R[T]/U)/(z)\cong k[D_{\tau}]$.

We denote by $P$ the ideal $(I_{D,A}:x_{\tau})$ of $R_{A}$, and by
\[%
\begin{tabular}
[c]{ll}%
$C_{P}:$ & $R_{A}/P\leftarrow P_{0}\overset{p_{1}}{\leftarrow}P_{1}%
\overset{p_{2}}{\leftarrow}\dots\overset{p_{g-1}}{\leftarrow}P_{g-1}%
\leftarrow0$%
\end{tabular}
\
\]
the minimal graded free resolution of $R/P$ as $R_{A}$-module. Moreover, we
denote by
\[%
\begin{tabular}
[c]{ll}%
$C_{z}:$ & $k[z]/(z)\leftarrow k[z]{\leftarrow}k[z]\leftarrow0$%
\end{tabular}
\
\]
the minimal graded free resolution of $k[z]/(z)$ as $k[z]$-module. Since
$J=(P,z)$ we have that $C_{J}$ is the tensor product (over $k$) of the
complexes $C_{P}$ and $C_{z}$. Hence $A_{0}=P_{0}^{a}$, $A_{g}=P_{g-1}^{a}$
and
\begin{equation}
A_{i}=P_{i-1}^{a}\oplus P_{i}^{a} \label{equation!decomp_of_Ai}%
\end{equation}
for all $1\leq i\leq g-1$, where $P_{i}^{a}=P_{i}\otimes_{k}k[z]$ considered
as $R$-module. Moreover, using this decomposition, we have that
\[
a_{1}=%
\begin{pmatrix}
p_{1} & z
\end{pmatrix}
,\quad a_{g}=%
\begin{pmatrix}
-z\\
p_{g-1}%
\end{pmatrix}
,\quad\text{ and }\quad\quad a_{i}=%
\begin{pmatrix}
p_{i} & -zE\\
0 & p_{i-1}%
\end{pmatrix}
\]
for $2\leq i\leq g-1$, where $E$ denotes the identity matrix of size equal to
the rank of $P_{i-1}$.

Recall from [loc. cit.] that the construction of $C_{U}$ involves chain maps
$\alpha:C_{I}\rightarrow C_{J}$, $\beta:C_{J}\rightarrow C_{I}[-1]$ and a
homotopy map $h:C_{I}\rightarrow C_{I}$, given by maps $\alpha_{i}%
:B_{i}\rightarrow A_{i}$, $\beta_{i}:A_{i}\rightarrow B_{i-1}$ and
$h_{i}:B_{i}\rightarrow B_{i}$ for all $i$. We will use that $\alpha_{0}$ is
an invertible element of $R$, that $h_{0}=h_{g}=0$, and that the $h_{i}$
satisfy the defining property
\begin{equation}
\beta_{i}\alpha_{i}=h_{i-1}b_{i}+b_{i}h_{i} \label{equ homotopy}%
\end{equation}
for all $i$.

Using the decomposition (\ref{equation!decomp_of_Ai}), we can write, for
$1\leq i\leq g-1$
\[
\alpha_{i}=%
\begin{pmatrix}
\alpha_{i,1}\\
\alpha_{i,2}%
\end{pmatrix}
,\quad\beta_{i}=%
\begin{pmatrix}
\beta_{i,1} & \beta_{i,2}%
\end{pmatrix}
\text{.}%
\]

\begin{proposition}
\label{prop!zeroForSomeComponents} We can choose $\alpha_{i},\beta_{i}$ and
$h_{i}$ in the following way:

\begin{enumerate}
\item $\alpha_{i},\beta_{i}$ do not involve $z$ for all $i$,

\item $\alpha_{i,2}=\beta_{i,1}=0$ for $1\leq i\leq g-1$, and

\item $h_{i}=0$ for all $i$.

\end{enumerate}
\end{proposition}

\begin{proof}
For the maps $\alpha_{i}$ the arguments are as follows. Since $\alpha_{0}$ is
an invertible element of $R$ it does not involve $z$. Assume now that $i=1$.
Using that $\alpha$ is a chain map, we have $\alpha_{0}b_{1}=a_{1}\alpha_{1}$,
hence
\[
\alpha_{0}b_{1}=%
\begin{pmatrix}
p_{1} & z
\end{pmatrix}%
\begin{pmatrix}
\alpha_{1,1}\\
\alpha_{1,2}%
\end{pmatrix}
=p_{1}\alpha_{1,1}+z\alpha_{1,2}\text{.}%
\]
Since $z$ does not appear in the product $\alpha_{0}b_{1}$ or in $p_{1}$ we
can assume $\alpha_{1,2}=0$ and that $z$ does not appear in $\alpha_{1,1}$.
Assume now that $\alpha_{i,2}=0$ and $\alpha_{i,1}$ does not involve the
variable $z$ and we will show that we can choose $\alpha_{i+1}$ with
$\alpha_{i+1,2}=0$ and that $z$ does not appear in $\alpha_{i+1,1}$. Indeed,
since $\alpha$ is a chain map, we have $\alpha_{i}b_{i+1}=a_{i+1}\alpha_{i+1}%
$, so
\[%
\begin{pmatrix}
\alpha_{i,1}\\
0
\end{pmatrix}
b_{i+1}=%
\begin{pmatrix}
p_{i+1} & -zE\\
0 & p_{i}%
\end{pmatrix}%
\begin{pmatrix}
\alpha_{i+1,1}\\
\alpha_{i+1,2}%
\end{pmatrix}
\text{.}%
\]
Hence we get the equations
\begin{equation}
\alpha_{i,1}b_{i+1}=p_{i+1}\alpha_{i+1,1}-z\alpha_{i+1,2},\quad0=p_{i}%
\alpha_{i+1,2}\text{.} \label{eqn!equationforaiplus1}%
\end{equation}
Write $\alpha_{i+1,1}=q_{1}+zq_{2}$ with $z$ not appearing in $q_{1}$.
Equation~(\ref{eqn!equationforaiplus1}) implies that $\alpha_{i,1}%
b_{i+1}=p_{i+1}q_{1}$. As a consequence, we can assume that $\alpha_{i+1,2}=0$
and that $\alpha_{i+1,1}=q_{1}$, hence $\alpha_{i+1,1}$ does not involve $z$.

For the maps $\beta_{i}$ the argument is as follows. Since $\psi(u)=0$ for all
$u\in P$ and $\psi(z)=x_{\tau}$, we have by \cite[Section~2]{BP3} that
$\beta_{1}=%
\begin{pmatrix}
0 & \dots & 0 & x_{\tau}%
\end{pmatrix}
$, hence $\beta_{1,1}=0$ and $z$ does not appear in $\beta_{1,2}$. Assume now
$\beta_{i,1}=0$ and $z$ does not appear in $\beta_{i,2}$ and we will show that
we can choose $\beta_{i+1}$ with $\beta_{i+1,1}=0$ and $z$ not appearing in
$\beta_{i+1,2}$. Indeed, since $\beta$ is a chain map, we have $b_{i}%
\beta_{i+1}=\beta_{i}a_{i+1}$, hence
\[
b_{i}%
\begin{pmatrix}
\beta_{i+1,1} & \beta_{i+1,2}%
\end{pmatrix}
=%
\begin{pmatrix}
0 & \beta_{i,2}%
\end{pmatrix}%
\begin{pmatrix}
p_{i+1} & -zE\\
0 & p_{i}%
\end{pmatrix}
\text{.}%
\]
Hence we get the equations
\[
b_{i}\beta_{i+1,1}=0,\quad\quad b_{i}\beta_{i+1,2}=\beta_{i,2}p_{i}%
\]
so we can assume that $\beta_{i+1,1}=0$ and that $z$ does not appear in
$\beta_{i+1,2}$.

We will now prove the statement for the maps $h_{i}$. Since, as proved above,
we can assume that $\alpha_{i,2}=\beta_{i,1}=0$, we have
\[
\beta_{i}\alpha_{i}=%
\begin{pmatrix}
0 & \beta_{i,2}%
\end{pmatrix}%
\begin{pmatrix}
\alpha_{i,1}\\
0
\end{pmatrix}
=0\text{.}%
\]
As a consequence, Equation~(\ref{equ homotopy}) can be satisfied by taking
$h_{i}=0$ for all $i$.
\end{proof}

In what follows, we will assume $\alpha_{i},\beta_{i}$ and $h_{i}$ are chosen
as in Proposition~\ref{prop!zeroForSomeComponents}.

\begin{proposition}
\label{prop!about_minimality_with_multidegrees} Assume that the face $\tau$ of
$D$ has the following property: every minimal non-face of $D$ contains at
least one vertex of $\tau$ (algebraically it means that for every minimal
monomial generator $v$ of $I$ there exists $p \in\tau$ such that $x_{p}$
divides $v$). Then $C_{U}$ is a minimal complex. As a consequence, we have
that $C_{U} \otimes_{R} R/(z)$ is, after substituting $T$ with $x_{j}$, the
minimal graded free resolution of $k[D_{\tau},2^{A \cup\{ j \}}]$.
\end{proposition}

\begin{proof}
We first show the minimality of $C_{U}$. Since we have $h_{i}=0$ for all $i$,
it is enough to show that, for $1\leq i\leq g-1$ the chain maps $\alpha_{i}$
and $\beta_{i}$ are minimal, in the sense that no nonzero constants appear in
the corresponding matrix representations. It follows by the defining
properties of the chain maps $\alpha$ and $\beta$ in \cite[Section~2]{BP3}
that $\beta_{i}$ is minimal if and only if $\alpha_{g-i}$ is. So it is enough
to prove that the map $\alpha_{i}\colon B_{i}\rightarrow A_{i}$ is minimal for
$1\leq i\leq g-1$. Denote by $M$ the monoid of exponent vectors on the
variables of $R$.

Since the ideals $I$ and $J$ of $R$ are monomial, there exist, for $1\leq
i\leq g-1$, positive integers $q_{1,i},q_{2,i}$ and multidegrees $\bar
{a}_{i,j_{1}},\bar{b}_{i,j_{2}}\in M$ with $1\leq j_{1}\leq q_{1,i}$ and
$1\leq j_{2}\leq q_{2,i}$ such that
\[
A_{i}=\bigoplus_{1\leq j_{1}\leq q_{1,i}}R(-\bar{a}_{i,j_{1}})\quad\text{ and
}\quad B_{i}=\bigoplus_{1\leq j_{2}\leq q_{2,i}}R(-\bar{b}_{i,j_{2}})\text{.}%
\]

For the minimality of $\alpha_{i}$ it is enough to show (compare
\cite[Remark~8.30]{MS}) that given $i$ with $1\leq i\leq g-1$ there are no
$j_{1},j_{2}$ with $1\leq j_{1}\leq q_{1,i}$, $1\leq j_{2}\leq q_{2,i}$ and
$\bar{a}_{i,j_{1}}=\bar{b}_{i,j_{2}}$, which we will now prove. By the
assumptions, given $v$ in the minimal monomial generating set of $I$ there
exists $p\in\tau$ with $x_{p}$ dividing $v$ in the polynomial ring $R$. Hence,
given $j_{2}$ with $1\leq j_{2}\leq q_{2,i}$ there is a nonzero coordinate of
$\bar{b}_{1,j_{2}}$ corresponding to a variable $x_{p}$ with $p\in\tau$. This
implies that the same is true for every $\bar{b}_{i,j_{2}}$ with $i\geq1$ and
$1\leq j_{2}\leq q_{2,i}$. On the other hand, no variable $x_{p}$ with
$p\in\tau$ appears in any minimal monomial generators of $J$, hence the same
is true for the coordinates of every $\bar{a}_{i,j_{1}}$ with $i\geq1$ and
$1\leq j_{1}\leq q_{1,i}$. So $\bar{a}_{i,j_{1}}=\bar{b}_{i,j_{2}}$ is
impossible for $i\geq1$. This finishes the proof that $C_{U}$ is a minimal complex.

By Proposition~\ref{prop!stellar_for_generalised} $z$ is $S$-regular and
$S/(z)\cong k[D_{\tau}]$. Hence using \cite[Proposition 1.1.5]{BH}, since
$C_{U}$ is minimal, the complex $C_{U}\otimes_{R}R/(z)$ is, after substituting
$T$ with $x_{j}$, the minimal graded free resolution of $k[D_{\tau}]$.
\end{proof}

\begin{remark}
\label{rem!minimalWithoutCondition} We give an example where the condition for
$\tau$ in the statement
Proposition~\ref{prop!about_minimality_with_multidegrees} is not satisfied but
$C_{U}$ is still minimal. Let $D$ be the simplicial complex triangulating the
$1$-dimensional sphere $S^{1}$ having $n$ vertices with $n\geq4$, and suppose
$\tau$ is a $1$-face of $D$. Since $n\geq4$ there exist minimal non-faces of
$D$ with vertex set disjoint from $\tau$. On the other hand $C_{U}$ is
minimal, see, for example, \cite[Section~5.2]{BP}.
\end{remark}

\section{Champions\label{sect!champions}}

\subsection{Construction}

Assume a positive integer $c\geq1$ is given. We will define a positive integer
$q$ and construct a simplicial complex $F_{c}\in\mathcal{D}_{q-1,c}$ such that
the inequalities of Theorem~\ref{thm!bettiBoundWeakVersion} are equalities.
First note that for $c=1$ we can take the boundary complex of any simplex, and
for $c=2$ any single stellar subdivision of that.

For $c\geq3$ we define inductively positive integers $d_{t}$, for $0\leq t\leq
c-1$, by $d_{0}=0$ and $d_{t+1}=d_{t}+(c-t)$, and set $q=d_{c-1}$. We also
define inductively, for $1\leq t\leq c-1$, subsets $\sigma_{t}\subset\lbrack
q]$ of cardinality $c$ by $\sigma_{1}=\{1,\dots,d_{1}=c\}$ and
\[
\sigma_{t+1}=\{(\sigma_{1})_{t},(\sigma_{2})_{t},\dots,(\sigma_{t})_{t}%
\}\cup\{i\bigm|d_{t}+1\leq i\leq d_{t+1}\}\text{,}%
\]
where $(\sigma_{i})_{p}$ denotes the $p$-th element of $\sigma_{i}$ with
respect to the usual ordering of $\mathbb{N}$. The main properties are that
$\#(\sigma_{i}\cap\sigma_{j})=1$ for all $i\not =j$, every three distinct
$\sigma_{i}$ have empty intersection, and the last element $d_{i}$ of
$\sigma_{i}$ is not in $\sigma_{j}$ for $j\not =i$.

\begin{example}
For $c=4$ we have $(d_{1},d_{2},d_{3})=(4,7,9)$, $q=9$, $\sigma_{1}%
=\{1,2,3,4\}$, $\sigma_{2}=\{1,5,6,7\}$ and $\sigma_{3}=\{2,5,8,9\}$. For
$c=5$ we have $(d_{1},\dots,d_{4})=(5,9,12,14)$, $q=14$, $\sigma
_{1}=\{1,2,3,4,5\}$, $\sigma_{2}=\{1,6,7,8,9\}$, $\sigma_{3}=\{2,6,10,11,12\}$
and $\sigma_{4}=\{3,7,10,13,14\}$. 

\end{example}

We define inductively simplicial subcomplexes $F_{t}\subset2^{[q+t-1]}$ for
$1\leq t\leq c$. Since $\sigma_{i}$ is not a subset of $\sigma_{j}$ for
$i\not =j$ we will be able to apply the elementary observation that if
$\sigma,\tau$ are two faces of a simplicial complex $D$ then $\tau$ not a
subset of $\sigma$ implies that $\sigma$ is also a face of the stellar
subdivision $D_{\tau}$. First set $F_{1}=\partial([q])\subset2^{[q]}$ to be
the boundary complex of the simplex on $q$ vertices $1,\dots,q$. Clearly
$\sigma_{i}$, for $1\leq i\leq c-1$, is a face of $F_{1}$. Set $F_{2}$ to be
the stellar subdivison of $F_{1}$ with respect to $\sigma_{1}$ with new vertex
$q+1$. Suppose $1\leq t\leq c-1$ and $F_{t}$ has been constructed. Since
$\sigma_{i}$ is a face of $F_{t}$ for $i\geq t$, we can continue inductively
and define $F_{t+1}$ to be the the stellar subdivision of $F_{t}$ with respect
to $\sigma_{t}$ with new vertex $q+t$.

The Stanley-Reisner ring of $F_{c}$ has the maximal possible Betti numbers
among all elements in $%
{\textstyle\bigcup\nolimits_{p\geq2}}
\mathcal{D}_{p,c}$:

\begin{proposition}
\label{prop!ChampionsProposition} For all $t$ with $1\leq t\leq c$ and all
$i\geq0$ we have
\[
b_{i}(R_{[q+t-1]}/I_{F_{t}})=l_{t,i}.
\]

\end{proposition}

We will give the proof in Subsection \ref{sec proof champion}.

\begin{remark}
Note that boundary complexes of stacked polytopes do not, in general, reach
the bounds.
\end{remark}

\begin{remark}
In the \textsc{Macaulay2} package \textsc{BettiBounds} \cite{BP4} we provide
an implementation of the construction of $F_{t}$. Using the minimality of the
Kustin-Miller complex, we also provide a function which produces their graded
Betti numbers. This works far beyond the range which is accessible by
computing the minimal free resolution via Gr\"{o}bner bases.
\end{remark}

\begin{example}
We use the implementation to produce $F_{4}$:\smallskip

\noindent%
\begin{tabular}
[c]{ll}%
\texttt{i1:} & \texttt{loadPackage "BettiBounds";}\smallskip
\end{tabular}

\noindent%
\begin{tabular}
[c]{ll}%
\texttt{i2:} & \texttt{F4 = champion 4;}%
\end{tabular}
\smallskip

\noindent%
\begin{tabular}
[c]{ll}%
\texttt{i3:} & \texttt{I4 = ideal F4}%
\end{tabular}
\smallskip

\noindent%
\begin{tabular}
[c]{ll}%
\texttt{o3:} & \texttt{ideal(x}$_{\mathtt{1}}$\texttt{x}$_{\mathtt{2}}%
$\texttt{x}$_{\mathtt{3}}$\texttt{x}$_{\mathtt{4}}$\texttt{,x}$_{\mathtt{1}}%
$\texttt{x}$_{\mathtt{5}}$\texttt{x}$_{\mathtt{6}}$\texttt{x}$_{\mathtt{7}}%
$\texttt{,x}$_{\mathtt{2}}$\texttt{x}$_{\mathtt{5}}$\texttt{x}$_{\mathtt{8}}%
$\texttt{x}$_{\mathtt{9}}$\texttt{,x}$_{\mathtt{5}}$\texttt{x}$_{\mathtt{6}}%
$\texttt{x}$_{\mathtt{7}}$\texttt{x}$_{\mathtt{8}}$\texttt{x}$_{\mathtt{9}}%
$\texttt{x}$_{\mathtt{10}}$\texttt{,x}$_{\mathtt{2}}$\texttt{x}$_{\mathtt{3}}%
$\texttt{x}$_{\mathtt{4}}$\texttt{x}$_{\mathtt{11}}$\texttt{,}\\
& \quad\quad\quad\texttt{x}$_{\mathtt{8}}$\texttt{x}$_{\mathtt{9}}$%
\texttt{x}$_{\mathtt{10}}$\texttt{x}$_{\mathtt{11}}$\texttt{,x}$_{\mathtt{1}}%
$\texttt{x}$_{\mathtt{3}}$\texttt{x}$_{\mathtt{4}}$\texttt{x}$_{\mathtt{12}}%
$\texttt{,x}$_{\mathtt{1}}$\texttt{x}$_{\mathtt{6}}$\texttt{x}$_{\mathtt{7}}%
$\texttt{x}$_{\mathtt{12}}$\texttt{, x}$_{\mathtt{6}}$\texttt{x}$_{\mathtt{7}%
}$\texttt{x}$_{\mathtt{10}}$\texttt{x}$_{\mathtt{12}}$\texttt{,x}%
$_{\mathtt{3}}$\texttt{x}$_{\mathtt{4}}$\texttt{x}$_{\mathtt{11}}$%
\texttt{x}$_{\mathtt{12}}$\texttt{,x}$_{\mathtt{10}}$\texttt{x}$_{\mathtt{11}%
}$\texttt{x}$_{\mathtt{12}}$\texttt{)}\smallskip
\end{tabular}

\noindent%
\begin{tabular}
[c]{ll}%
\texttt{i4:} & \texttt{betti res I4}%
\end{tabular}
\smallskip

\noindent%
\begin{tabular}
[c]{rrrrrrr}
&  & \texttt{0} & \texttt{1} & \texttt{2} & \texttt{3} & \texttt{4}\\
\texttt{o4:} & \texttt{total:} & \multicolumn{1}{l}{\texttt{1}} &
\multicolumn{1}{l}{\texttt{11}} & \multicolumn{1}{l}{\texttt{20}} &
\multicolumn{1}{l}{\texttt{11}} & \multicolumn{1}{l}{\texttt{1}}\\
& \texttt{0:} & \multicolumn{1}{l}{\texttt{1}} & \multicolumn{1}{l}{\texttt{.}%
} & \multicolumn{1}{l}{\texttt{.}} & \multicolumn{1}{l}{\texttt{.}} &
\multicolumn{1}{l}{\texttt{.}}\\
& \texttt{1:} & \multicolumn{1}{l}{\texttt{.}} & \multicolumn{1}{l}{\texttt{.}%
} & \multicolumn{1}{l}{\texttt{.}} & \multicolumn{1}{l}{\texttt{.}} &
\multicolumn{1}{l}{\texttt{.}}\\
& \texttt{2:} & \multicolumn{1}{l}{\texttt{.}} & \multicolumn{1}{l}{\texttt{1}%
} & \multicolumn{1}{l}{\texttt{.}} & \multicolumn{1}{l}{\texttt{.}} &
\multicolumn{1}{l}{\texttt{.}}\\
& \texttt{3:} & \multicolumn{1}{l}{\texttt{.}} & \multicolumn{1}{l}{\texttt{9}%
} & \multicolumn{1}{l}{\texttt{9}} & \multicolumn{1}{l}{\texttt{1}} &
\multicolumn{1}{l}{\texttt{.}}\\
& \texttt{4:} & \multicolumn{1}{l}{\texttt{.}} & \multicolumn{1}{l}{\texttt{.}%
} & \multicolumn{1}{l}{\texttt{2}} & \multicolumn{1}{l}{\texttt{.}} &
\multicolumn{1}{l}{\texttt{.}}\\
& \texttt{5:} & \multicolumn{1}{l}{\texttt{.}} & \multicolumn{1}{l}{\texttt{1}%
} & \multicolumn{1}{l}{\texttt{9}} & \multicolumn{1}{l}{\texttt{9}} &
\multicolumn{1}{l}{\texttt{.}}\\
& \texttt{6:} & \multicolumn{1}{l}{\texttt{.}} & \multicolumn{1}{l}{\texttt{.}%
} & \multicolumn{1}{l}{\texttt{.}} & \multicolumn{1}{l}{\texttt{1}} &
\multicolumn{1}{l}{\texttt{.}}\\
& \texttt{7:} & \multicolumn{1}{l}{\texttt{.}} & \multicolumn{1}{l}{\texttt{.}%
} & \multicolumn{1}{l}{\texttt{.}} & \multicolumn{1}{l}{\texttt{.}} &
\multicolumn{1}{l}{\texttt{.}}\\
& \texttt{8:} & \multicolumn{1}{l}{\texttt{.}} & \multicolumn{1}{l}{\texttt{.}%
} & \multicolumn{1}{l}{\texttt{.}} & \multicolumn{1}{l}{\texttt{.}} &
\multicolumn{1}{l}{\texttt{1}}%
\end{tabular}
\smallskip

The command \texttt{gradedBettiChampion 20}, will produce the Betti table of
the minimal free resolution of $I_{F_{20}}$ with projective dimension $20$ and
regularity $208$ in $0.7$ seconds\footnote{On a singe core of an Intel
i7-2640M at $3.4$ GHz.}. For more examples, see the documentation of
\textsc{BettiBounds}.
\end{example}

\subsection{Proof of Proposition \ref{prop!ChampionsProposition}%
\label{sec proof champion}}

The main idea of the proof is that when passing from $F_{t}$ to $F_{t+1}$ by
subdividing $\sigma_{t}$, the ideals $I_{F_{t}}$ and $(I_{F_{t}}:x_{\sigma
_{t}})$ have the same total Betti numbers (Proposition
\ref{prop!keyPropositionForChampions}) and the Kustin-Miller complex
construction yields a minimal free resolution (Lemma
\ref{prop!keyMinimalityOfKMComplexConstruction}).

It is convenient to introduce the following notations, which will be used only
in the present subsection. For nonzero monomials $v=\prod_{i=1}^{l}%
x_{i}^{a_{i}}$ and $w=\prod_{i=1}^{l}x_{i}^{b_{i}}$ in $R_{[l]}$ we set
\[
\frac{v}{w}=\prod_{i=1}^{l}x_{i}^{c_{i}},\quad\text{ with }c_{i}=\max
(a_{i}-b_{i},0)\text{,}%
\]
and for a set $S$ of monomials we set $\frac{S}{w}=\{\frac{v}{w}\bigm|v\in
S\}$. Clearly $\frac{v}{w}$ is the monomial generator of the ideal quotient
$((v):(w))$.

For simplicity of notation write $\mathcal{T}=R_{[q+t-1]}$. We will now study
in more detail the the Stanley--Reisner ideal $I_{F_{t}}\subset\mathcal{T}$ of
$F_{t}$. We set $u_{1}=\prod_{i=1}^{q}x_{i}$, $u_{2}=\frac{x_{q+1}u_{1}%
}{x_{\sigma_{1}}}$ and inductively define finite subsets $S_{t}\subset
I_{F_{t}}$ by $S_{1}=\{u_{1}\},S_{2}=\{u_{2},x_{\sigma_{1}}\}$ and, for
$t\geq2$,
\begin{equation}
S_{t+1}=S_{t}\;\cup\;\frac{x_{q+t}S_{t}}{x_{\sigma_{t}}}\;\cup\;\{x_{\sigma
_{t}}\}\text{.}\label{eqn!givingStplus1}%
\end{equation}
Clearly $S_{1}$ (resp. $S_{2}$) is the minimal monomial generating set of
$I_{F_{1}}$ (resp. $I_{F_{2}}$). Moreover, an easy induction on $t$ using
Equation~(\ref{eqn!egns_of_stellar}) shows that $S_{t}$ is a set of monomials
generating $I_{F_{t}}$ for all $1\leq t\leq c$. In
Proposition~\ref{prop!keyPropositionForChampions} we will show that $S_{t}$ is
actually the minimal monomial generating set of $I_{F_{t}}$ for all $t$.

Equation~(\ref{eqn!givingStplus1}) and induction imply that given an element
$v$ of $S_{t+1}$ there exists $e_{v}\in\{u_{2},x_{\sigma_{1}},\dots
,x_{\sigma_{t}}\}$ such that either $v=e_{v}$ or $v=\frac{w_{1}e_{v}}{w_{2}}$,
with $w_{1}=\prod_{j=1}^{l}x_{q+r_{j}}$ and $w_{2}=\prod_{j=1}^{l}%
x_{\sigma_{r_{j}}}$ for some $l\geq1$ and $r_{1}<r_{2}<\dots<r_{l}\leq t$.
Moreover, if $e_{v}=u_{2}$ we have $2\leq r_{1}$, while if $e_{v}%
=x_{\sigma_{p}}$ we have $p+1\leq r_{1}$. A priori $e_{v}$ may not be uniquely
determined and we fix one of them and call it the original source of $v$. One
can actually show that in our setting $e_{v}$ is uniquely determined by $v$
but we do not prove it and do not use it in the following.

\begin{example}
We have
\[
S_{3}=\left\{  u_{2},\text{ }\frac{x_{q+2}u_{2}}{x_{\sigma_{2}}}\right\}
\;\cup\;\left\{  x_{\sigma_{1}},\text{ }\frac{x_{q+2}x_{\sigma_{1}}}%
{x_{\sigma_{2}}}\right\}  \;\cup\;\left\{  x_{\sigma_{2}}\right\}
\]
and
\[
S_{4}=\left\{  u_{2},\frac{\text{ }x_{q+2}u_{2}}{x_{\sigma_{2}}},\text{ }%
\frac{x_{q+3}u_{2}}{x_{\sigma_{3}}},\text{ }\frac{x_{q+2}x_{q+3}u_{2}%
}{x_{\sigma_{2}}x_{\sigma_{3}}}\right\}  \;\cup
\]%
\[
\phantom{=======}\left\{  x_{\sigma_{1}},\text{ }\frac{x_{q+2}x_{\sigma_{1}}%
}{x_{\sigma_{2}}},\text{ }\frac{x_{q+3}x_{\sigma_{1}}}{x_{\sigma_{3}}},\text{
}\frac{x_{q+2}x_{q+3}x_{\sigma_{1}}}{x_{\sigma_{2}}x_{\sigma_{3}}}\right\}
\;\cup
\]%
\[
\left\{  x_{\sigma_{2}},\text{ }\frac{x_{q+3}x_{\sigma_{2}}}{x_{\sigma_{3}}%
}\right\}  \cup\left\{  x_{\sigma_{3}}\right\}  \text{.}%
\]

\end{example}

We now fix $t$ with $t\leq c-1$. Part (1) of the following combinatorial lemma
will be used in Lemma~\ref{prop!keyMinimalityOfKMComplexConstruction} for the
proof of the minimality of the Kustin--Miller complex construction, while part
(2) will be used in Proposition~\ref{prop!keyPropositionForChampions} for the
proof of the equality of the corresponding Betti numbers of $\mathcal{T}%
/I_{F_{t}}$ and $\mathcal{T}/(I_{F_{t}}:x_{\sigma_{t}})$.

\begin{lemma}
\label{prop!keyCombinatorialForChampions}

\begin{enumerate}
\item For every $v\in S_{t}$ there exists $a\in\sigma_{t}$ such that $x_{a}$
divides $v$.

\item We can recover $S_{t}$ from $\frac{S_{t}}{x_{\sigma_{t}}}$ in the
following way: $S_{t}$ is the set obtained from $\frac{S_{t}}{x_{\sigma_{t}}}$
by substituting, for $p=1,2,\dots,t-1$, the variable $x_{(\sigma_{p})_{t}}$ by
the product $\;x_{(\sigma_{p})_{t}}x_{(\sigma_{p})_{t}-1}$,
and substituting the variable $x_{d_{t}+1}$ by the product $\prod
_{r=d_{t-1}+1}^{d_{t}+1}x_{r}$.
\end{enumerate}
\end{lemma}

\begin{proof}
Let $v\in S_{t}$ and consider the original source $e_{v}\in\{u_{2}%
,x_{\sigma_{1}},\dots,x_{\sigma_{t-1}}\}$ of $v$. Write
\begin{equation}
v=\frac{w_{1}e_{v}}{w_{2}}\text{,} \label{eqn!expressionForv}%
\end{equation}
with either $(w_{1}=w_{2}=1)$ or $w_{1}=\prod_{j=1}^{l}x_{q+r_{j}}$ and
$w_{2}=\prod_{j=1}^{l}x_{\sigma_{r_{j}}}$ for some $l\geq1$ and $r_{1}%
<r_{2}<\dots<r_{l}\leq t-1$. Moreover, if $e_{v}=u_{2}$ we have $2\leq r_{1}$,
while if $e_{v}=x_{\sigma_{p}}$ we have $p+1\leq r_{1}$.

We first prove (1). If $e_{v}=u_{2}$, we set $a=d_{t}\in\sigma_{t}$ and
observe that $x_{a}$ divides $e_{v}$. Since $d_{t}$ is not in any $\sigma_{i}$
for $i<t$ we have that $x_{a}$ does not divide $w_{2}$, hence it follows by
(\ref{eqn!expressionForv}) that $x_{a}$ divides $v$. Assume now that
$e_{v}=\sigma_{p}$ for some $p$ with $1\leq p\leq t-1$. We set $a=(\sigma
_{p})_{t-1}$. By the definition of the sets $\sigma_{r}$, we have that $a$ is
in the intersection of $\sigma_{p}$ with $\sigma_{t}$ and in no other
$\sigma_{r}$. Hence $x_{a}$ divides $e_{v}$ but not $w_{2}$, hence it follows
by (\ref{eqn!expressionForv}) that $x_{a}$ divides $v$.

We will now prove (2). We first fix $p\in\{1,2,\dots,t-1\}$, set
$m=(\sigma_{p})_{t}$, assume $x_{m}$ divides $v$, and prove that $x_{m-1}$
also divides $v$. The assumption that $x_{m}$ divides $v$ implies that, when
$v\not =e_{v}$, in the expression (\ref{eqn!expressionForv}) we have
$r_{i}\not =p$ for $1\leq i\leq l$. Taking into account that $m$ is not in
$\sigma_{j}$ for $1\leq j\leq t-1$ and $j\not =p$ we get that $e_{v}=u_{2}$ or
$e_{v}=\sigma_{p}$. Since $p<t$ we have $m-1=(\sigma_{p})_{t-1}$. This,
together with $e_{v}\in\{u_{2},\sigma_{p}\}$ implies that $x_{m-1}$ divides
$e_{v}$. It also implies that $m-1$ is not in any $\sigma_{j}$ for $1\leq
j\leq t-1$ and $j\not =p$. Hence, $x_{m-1}$ does not divide $w_{2}$ and since
it divides $e_{v}$ if follows from (\ref{eqn!expressionForv}) that it also
divides $v$.

We now assume $x_{d_{t}+1}$ divides $v$ and will show that $\prod
_{r=d_{t-1}+1}^{d_{t}+1}x_{r}$ also divides $v$. Since $d_{t}+1$ is not in
$\sigma_{i}$ for $1\leq i\leq t-1$, we have that $e_{v}=u_{2}$. Fix $r$ with
$d_{t-1}+1\leq r\leq{d_{t}}$. Then $r$ is not an element of $\sigma_{j}$ for
$1\leq j\leq t-1$. Hence $x_{r}$ does not divide $w_{2}$ and since it divides
$u_{2}$ if follows from (\ref{eqn!expressionForv}) that it also divides $v$.
Taking into account that $m$ and $d_{t}+1$ are not in $\sigma_{t}$, this
completes the proof of (2).
\end{proof}

\begin{lemma}
\label{prop!keyMinimalityOfKMComplexConstruction} Fix $t$ with $2\leq t\leq
c-1$. Then the Kustin--Miller complex construction related to the unprojection
pair $(I_{F_{t}}:x_{\sigma_{t}},z)\subset\mathcal{T}[z]/(I_{F_{t}})$ and using
as initial data the minimal graded free resolutions of $\mathcal{T}%
[z]/(I_{F_{t}})$ and $\mathcal{T}[z]/(I_{F_{t}}:x_{\sigma_{t}},z)$ gives a
minimal complex.
\end{lemma}

\begin{proof}
The minimal monomial generating set of $I_{F_{t}}$ is a subset, say $\tilde
{S}_{t}$, of $S_{t}$. By part (1) of
Lemma~\ref{prop!keyCombinatorialForChampions} given $v\in\tilde{S}_{t}$, there
is an $a\in\sigma_{t}$ with $x_{a}$ dividing $v$. As a consequence, the result
follows from Proposition~\ref{prop!about_minimality_with_multidegrees}.
\end{proof}

\begin{proposition}
\label{prop!keyPropositionForChampions} Fix $t$ with $2 \leq t \leq c$. Then

\begin{enumerate}
\item The set $S_{t}$ is the minimal monomial generating set of $I_{F_{t}}$.

\item The corresponding Betti numbers of $\mathcal{T}/I_{F_{t}}$ and
$\mathcal{T}/(I_{F_{t}}:x_{\sigma_{t}})$ are equal, that is
\[
b_{i}(\mathcal{T}/(I_{F_{t}}:x_{\sigma_{t}}))=b_{i}(\mathcal{T}/I_{F_{t}})
\]
for all $i$. In particular, the set $\frac{S_{t}}{x_{\sigma_{t}}}$ has the
same cardinality as $S_{t}$ and is the minimal monomial generating set of
$(I_{F_{t}}:x_{\sigma_{t}})$.
\end{enumerate}
\end{proposition}

\begin{proof}
We use induction on $t$. For $t=2$ we have that both $I_{F_{t}}$ and
$(I_{F_{t}}:x_{\sigma_{t}})$ are codimension $2$ complete intersections, so
both (1) and (2) are obvious. Assume that (1) and (2) are true for a value
$t<c-1$ and we will show that they are true also for the value $t+1$. By
Lemma~\ref{prop!keyMinimalityOfKMComplexConstruction} the Kustin--Miller
complex construction related to the unprojection pair $(I_{F_{t}}%
:x_{\sigma_{t}},z)\subset\mathcal{T}[z]/I_{F_{t}}$ and using as input data the
minimal graded free resolutions of $\mathcal{T}[z]/I_{F_{t}}$ and
$\mathcal{T}[z]/(I_{F_{t}}:x_{\sigma_{t}},z)$ gives a minimal complex. In
particular, this implies that $S_{t+1}$ is the minimal monomial generating set
of $I_{F_{t+1}}$.

We now look more carefully the substitutions in part (2) of
Lemma~\ref{prop!keyCombinatorialForChampions}. Assume $p\leq t$ and set
$m=(\sigma_{p})_{t+1}$. Since $p<t+1$ we have by the construction of
$\sigma_{p}$ that $m-1=(\sigma_{p})_{t}$, so $m-1$ is an element of
$\sigma_{t+1}$. Consequently $x_{m-1}$ does not appear as variable in
$\frac{S_{t+1}}{x_{\sigma_{t+1}}}$. Similarly, for each $r$ with $d_{t}+1\leq
r\leq d_{t+1}$ we have $r\in\sigma_{t+1}$, so $x_{r}$ does not appear as
variable in $\frac{S_{t+1}}{x_{\sigma_{t+1}}}$. Using these facts, the
equality of Betti numbers in part (2) follows by arguing as in the proof of
\cite[Proposition 6.5]{BP2}. Since we have shown that $S_{t+1}$ is the minimal
monomial generating set of $I_{F_{t+1}}$, and $\frac{S_{t+1}}{x_{\sigma_{t+1}%
}}$ contains the minimal monomial generating set of $(I_{F_{t}}:x_{\sigma_{t}%
})$, the equality of Betti numbers we just showed implies, for $i=1$, that
$\frac{S_{t+1}}{x_{\sigma_{t+1}}}$ has the same cardinality as $S_{t+1}$ and
is the minimal monomial generating set of the ideal $(I_{F_{t+1}}%
:x_{\sigma_{t+1}})$.
\end{proof}

We now give the proof of Proposition \ref{prop!ChampionsProposition}.

\begin{proof}
The proof is by induction on $t$. For $t=1,2$ the result is clear. Assume that
the result is true for some value $2\leq t\leq c-1$ and we will show it is
true for $t+1$. We set for simplicity $A_{1}=\mathcal{T}/I_{F_{t}}$ and
$A_{2}=\mathcal{T}/(I_{F_{t}}:x_{\sigma})$.

By the inductive hypothesis $b_{i}(A_{1})=l_{t,i}$ and by part (2) of
Proposition~\ref{prop!keyPropositionForChampions} $b_{i}(A_{2})=b_{i}(A_{1})$,
hence $b_{i}(A_{2})=b_{i}(A_{1})=l_{t,i}$ for all $i$. Since by
Lemma~\ref{prop!keyMinimalityOfKMComplexConstruction} the corresponding
Kustin--Miller construction is minimal, we get that
\[
b_{1}(R_{[q+t]}/I_{F_{t+1}})=b_{1}(A_{1})+b_{1}(A_{2})+1=2l_{t,1}+1=l_{t+1,1}%
\]
and that for $i$ with $2\leq i\leq\operatorname{codim}R_{[q+t]}/(I_{F_{t+1}%
})-2$
\begin{align*}
b_{i}(R_{[q+t]}/I_{F_{t+1}})  &  =b_{i-1}(A_{1})+b_{i}(A_{1})+b_{i-1}%
(A_{2})+b_{i}(A_{2})\\
&  =2l_{t,i-1}+2l_{t,i}=l_{t+1,i}%
\end{align*}
which finishes the proof.
\end{proof}

\emph{Acknowledgements.} The authors would like to thank an anonymous referee
of \cite{BP} for suggesting to do research in this direction, and
Micha\l \ Adamaszek, J\"{u}rgen Herzog and Diane Maclagan for helpful comments.

\end{document}